\documentclass[10pt,a4paper]{amsart}
\usepackage{amscd,amssymb,amsopn,amsmath,amsthm,graphics,amsfonts,enumerate,verbatim,calc}
\usepackage[dvips]{graphicx}

\usepackage{amscd,amssymb,amsopn,amsmath,amsthm,graphics,amsfonts,enumerate,verbatim,calc}
\usepackage[dvips]{graphicx}
\input xy
\xyoption{all}

\usepackage{amssymb,amsmath}

\headsep=1cm \numberwithin{equation}{section}
\newtheorem{theorem}{Theorem}[section]
\newtheorem{lemma}[theorem]{Lemma}

\newtheorem{proposition}[theorem]{Proposition}
\newtheorem{corollary}[theorem]{Corollary}

\theoremstyle{definition}
\newtheorem{definition}[theorem]{Definition}
\theoremstyle{remark}
\newtheorem{remark}[theorem]{Remark}
\newtheorem{example}[theorem]{Example}

\newtheorem{fact}[theorem]{Fact}
\newtheorem{question}[theorem]{Question}

\newtheorem{acknowledgement}{Acknowledgement}

\newcommand{\Ass}{\operatorname{Ass}}

\newcommand{\sym}{\operatorname{Sym}}

\newcommand{\Kgrade}{\operatorname{K.grade}}

\newcommand{\Spec}{\operatorname{Spec}}

\newcommand{\Gr}{\operatorname{Gr}}
\newcommand{\Mat}{\operatorname{Mat}}

\newcommand{\pd}{\operatorname{p.dim}}
\newcommand{\gd}{\operatorname{gl.dim}}

\newcommand{\Wdim}{\operatorname{w.dim}}
\newcommand{\fd}{\operatorname{fl.dim}}

\newcommand{\Tor}{\operatorname{Tor}}

\newcommand{\Hom}{\operatorname{Hom}}

\newcommand{\Char}{\operatorname{char}}

\newcommand{\lo}{\longrightarrow}
\newcommand{\fm}{\frak{m}}
\newcommand{\fp}{\frak{p}}

\newcommand{\fa}{\frak{a}}
\newcommand{\fb}{\frak{b}}

\newcommand{\fn}{\frak{n}}

\begin{document}
\author[]{Mohsen Asgharzadeh }

\title[]
{Desingularization of regular algebras}

\address{}
\email{mohsenasgharzadeh@gmail.com}



\begin{abstract}
We identify families of commutative rings
that can be written as a direct limit of a directed system of noetherian
regular rings and  investigate the
homological properties of such  rings.
\end{abstract}

\maketitle

\section{Introduction}

The goal of this work is to identify  rings  $R$
that can be realized as a direct limit of a directed system $\{R_i:i\in \Gamma\}$ of noetherian
regular rings (which we then call a \textit{densingularization} of $R$), and to investigate the
homological properties of such an $R$.
We emphasize that the poset $\Gamma$ is filtered. A paradigm for this, and one of the motivation
for this work, is a result of Zariski \cite{Z} (and Popescu \cite{P}):

\begin{theorem}\label{zar}(Zariski-Popescu)
Let $(V,\fm)$ be a valuation domain containing a field  $k$ of zero characteristic. Then $V$ has a densingularization.
\end{theorem}

It may be interesting to mention that  the construction of densingularizations  goes back to Akizuki \cite{Ak} and Nagata \cite{N}.
Recall from \cite{Ber} that a ring is said to be \textit{regular}, if each   finitely generated ideal has finite projective dimension.
A ring is called \textit{coherent}, if its finitely generated ideals are finitely presented.
Our first result in  Section 2 is:

\begin{proposition}
Let $R$ be a ring that has a desingularization and $\fp$ a finitely generated prime
ideal in $R$.
 If $R_{\fp} $  is coherent, then $R_{\fp} $ is regular.
\end{proposition}

 Also,  Section 2 is devoted to computing the
homological dimensions of an ideal $I$ of  a ring with  a desingularization $\{R_i:i\in \Gamma\}$.
We do this by imposing some additional assumptions both on the ideal $I$, the rings $R_i$ and the poset $\Gamma$.

A  \textit{quasilocal} ring  is a ring with a unique maximal ideal. A  local
ring  is a noetherian quasilocal ring.
There are many definitions  for the regularity condition  in non-noetherian rings (see e.g. \cite{K}).
One of these is the notion of \textit{super regularity}. This notion was first
introduced  by Vasconcelos  \cite{V1}. A  coherent  quasilocal ring  is called super regular  if  its global dimension
is finite and equal to its weak dimension.
Section 3 deals with a desingularization of super regular rings.
Our first result in this direction is
Proposition \ref{kab}:

\begin{proposition}
Let $\{(R_i,\fm_i)\}$ be a directed system of  local rings with the property that
 $\fm_i^2=\fm_i\cap \fm_{i+1}^2$. If $R:={\varinjlim}R_i$ is coherent and super regular, then each  $ R_i $  is regular.
\end{proposition}

We present a nice application of the notion of super regularity: we compute the global dimension of certain \textit{perfect} algebras. To this end,  suppose
 $R$  is a complete local  domain which is not  a field and
  suppose that its \textit{perfect closure} $R^{\infty}$ is coherent.  In Proposition \ref{rem}  we show that $$\gd(R^{\infty})=\dim R+1.$$

 Let $\{R_i\}$ be a pure directed system of local rings and  suppose that the maximal ideal of
$R:={\varinjlim}_{}R_i$  has a finite free resolution.
 In Proposition \ref{cri}, we show $R$ is noetherian and  regular.
Let $F$ be a finite field. In Proposition \ref{pro}  we present the desingularization  of  $\prod_{\mathbb{N}} F$.
This has some applications. For example,
 $\prod_{\mathbb{N}} F$ is stably coherent.

We cite \cite{G} as a reference book on  commutative coherent rings.

\section{Homological properties of  a desingularization}

We start by introducing some notation.
 By $\pd_R(-)$ (resp. $\fd_R(-)$), we mean
projective dimension (resp. flat dimension)
of an $R$-module. Denote
the $i^{th}$ Koszul homology  module of $R$ with respect to $\underline{x}:=x_1,\ldots,x_n$
by $H_i( \underline{x}; R)$.

\begin{remark} \label{kk}
Let $\{R_i:i\in \Gamma\}$ be a directed  system of  rings and let $\underline{x}:=x_1,\ldots,x_n$ be
in $R:={\varinjlim}R_i$. Let $i_0$ be such that $\underline{x}\subset R_{i_0}$. Then  ${\varinjlim}_{i\geq i_0}H_{\bullet}(\underline{x}, R_i)\simeq H_{\bullet}(\underline{x}, R)$.
\end{remark}

\begin{proof}
This is straightforward and we leave it to the reader.
\end{proof}

\begin{definition}\label{defh}Let $(R,\fm)$ be a quasilocal ring. Suppose $\fm$ is generated by
a finite sequence of elements $\underline{x}:=x_1,\ldots,x_n$.
Recall from Kabele \cite{K} that $R$ is \textit{$H_1$-regular}, if  $H_1( \underline{x}, R)=0$.
Also, $R$ is called \textit{ Koszul regular}, if  $H_i(\underline{x}, R)=0$ for all $i>0$.
\end{definition}

In general, $H_1$-regular rings are not Koszul regular, see \cite{K}.

\begin{lemma} \label{sy2} Let $(R,\fm)$ be a coherent  $H_1$-regular ring. Then $R$ is  Koszul regular.
\end{lemma}

\begin{proof} Coherence regular local
rings are integral domains.
Recall from Definition \ref{defh} that  $\fm$ is finitely generated.
 Let $\underline{x}:=x_1,\ldots,x_n$ be a generating set for  $\fm$.
Since $R$ is coherent and in view of \cite[Lemma 3.7]{AT}, the $R$-module $H_i(\underline{y},R)$
is finitely generated, where $\underline{y}$ is a finite sequence of elements. By using basic properties of Koszul homologies
and  by an easy induction
 one may show that $H_i(x_1,\ldots,x_j; R)=0$ for all $i>0$ and all $j$. We left the routine details to the reader, please see \cite[Page 127-128 ]{Mat}.
\end{proof}

\begin{proposition} \label{sy3} Let $R$ be a ring with a desingularization and let $\fp\in\Spec( R)$ be finitely generated.
 If $R_{\fp} $  is coherent, then $R_{\fp} $ is regular.
\end{proposition}

\begin{proof}
 Let $\{R_i:i\in \Gamma\}$ be a directed system of noetherian regular rings
such that $R:={\varinjlim}R_i$. To simplify the notation, we replace $R_{\fp} $ with $(R,\fm)$ and  $(R_i)_{\fp\cap R_i}$ with $(R_i,\fm_i)$.
In view of the natural isomorphism $R_{\fp}\simeq{\varinjlim}(R_i)_{\fp\cap R_i},$  we may do such a replacement.
 Let $\underline{x}:=x_1,\ldots,x_n$ be a generating set for  $\fm$. Without loss of the generality, we can
 assume that $x_i\in R_j$ for all $i$ and all $j$.
Set $A_i:=R_i/(\underline{x})$. In the light of \cite[Lemma 2.5.1]{MR},\[\begin{array}{ll}
 $$0 \to H_2(R_i,A_i,A_i)\to H_1(\underline{x},R_i) \to A_i^n\to(\underline{x})/(\underline{x})^2 \to 0,\end{array}\]
where $H_\ast(-,-,-)$ is the \textit{Andr\'{e}-Quillen homology}. By Remark \ref{kk},
 Koszul homology behaves well
with   respect to direct limits. Recall from \cite[Proposition 1.4.8]{MR}   that Andr\'{e}-Quillen homology behaves well
with  respect to direct limits.  These induce  the following exact sequence 
$$0 \lo H_2(R,k,k) \lo H_1(\underline{x},R) \lo R^n/\fm R^n \stackrel{\pi}\lo\fm/\fm^2\lo 0,$$
where the map $\pi$ induced from the natural surjective
homomorphism that sends the canonical basis of $A_i^n$ to $\underline{x}$. We view  $\pi$ as a surjective map of finite dimensional
vector spaces with the same dimension. In particular,  $\pi$  is an isomorphism.

Recall that $k$ (resp. $k_i$) is the residue field of $R$ (resp. $R_i$). In the light  of \cite[Proposition 1.4.8, Corollary 2.5.3]{MR}, $H_2(R,k,k)\simeq{\varinjlim}H_2(R_i,k_i,k_i)=0$. Thus $H_1(\underline{x},R)=0$,
because $\pi$ is an isomorphism.  Due to Lemma \ref{sy2}, $\fd_R(k)<\infty$.
Again, as $R$ is coherent and in view of \cite[Corollary 2.5.10]{G},  any finitely generated ideal of $R$
 has finite projective dimension, i.e., $R$ is regular.
\end{proof}

By $\Wdim(R)$, we mean the \textit{weak dimension} of $R$.   By definition
$$ \Wdim (R):=\sup\{\fd(M):M   \textit{ is an }  R\textit{-module}  \},$$  see \cite[Page 20]{G}.

\begin{lemma} \label{sy} Let $\{R_i:i\in \Gamma\}$ be a directed system of
rings such that their weak dimension is bounded by an integer $n$. Set $R:={\varinjlim}R_i$. The following assertions hold:
\begin{enumerate}
\item[$\mathrm{(i)}$]  The flat dimension of an
 $R$-module   is bounded above by  $n$.
\item[$\mathrm{(ii)}$] If $R$ is coherent, then projective dimension of any
finitely presented $R$-module is bounded above by  $n$.
\end{enumerate}
\end{lemma}

\begin{proof} $\mathrm{(i)}$: Let $M$ and $N$ be two $R$-modules. By \cite[VI, Exercise 17]{CE},
$\Tor_j^{R}(M,N)\simeq\underset{i}{\varinjlim}\Tor_j^{R_i}(M,N),$
which is zero for all $j>n$ and this is the thing that we search for.

$\mathrm{(ii)}$:
This follows by \cite[Corollary 11.5]{St}.
\end{proof}

Let $\fa$ be an ideal of a ring $R$ and $M$ an
$R$-module. Let $\Sigma$ be the family of all finitely generated
subideals $\fb$ of $\fa$. The \textit{Koszul grade} of a finitely generated ideal $\frak a:=(x_1,\ldots,x_n)$ on $M$
is defined by
$$\Kgrade_R(\fa,M):=\inf\{i \in\mathbb{N}\cup\{0\} | H^{i}(\Hom_R(
K_{\bullet}(\underline{x}), M)) \neq0\}.$$ Note that by
\cite[Corollary 1.6.22]{BH} and \cite[Proposition 1.6.10 (d)]{BH},
this does not depend on the choice of generating sets of $\fa$. For
an ideal $\frak a$ (not necessarily finitely generated), Koszul grade of
$\fa$ on $M$ can be defined by
$\Kgrade_R(\fa,M):=\sup\{\Kgrade_R(\fb,M):\fb\in\Sigma\}.$ By using
\cite[Proposition 9.1.2 (f)]{BH}, this definition coincides with the
original definition for finitely generated ideals.

\begin{corollary}\label{coor}
Let $\{R_i:i\in \Gamma\}$ be a directed system of coherent  regular quasilocal rings such that their
 Krull dimension is bounded by an integer. Suppose    each $R_i$ is noetherian and $\Gamma$ is countable, or
$R$ is  coherent. Then $R:={\varinjlim}R_i$ is regular.
\end{corollary}

\begin{proof}
First, suppose that  each $R_i$ is noetherian and $\Gamma$ is countable. Any ideal of $R$ is countably generated.
It follows by the proof of \cite[Corollary 2.47]{O1},  that $\pd(-)\leq \fd(-)+1$. It remains
  to recall $\Wdim(R_i)=\dim (R_i)$, because $R_i$ is noetherian.

Now, suppose that $R$ is coherent. Let $I$ be a finitely generated ideal of $R$ generated by $\underline{x}:=x_1,\ldots,x_n$. There is  $j\in \Gamma$ such that
$\underline{x}\subseteq R_i$ for all $i\geq j$. Denote $\underline{x}R_i$ by $I_i$ and define $\fm_i:=\fm\cap R_i$. In view of  \cite[Lemma 3.2]{AT}, $\Kgrade_{R_i}(\fm_i,R_i)\leq\dim R_i$.
Note that $R_i/I_i$ has a finite free resolution.
By \cite[Chap. 6, Theorem 2]{No}, \[\begin{array}{ll}
\fd(R_i/I_i)&\leq\pd(R_i/I_i)\\
&=\Kgrade(\fm_i,R_i)-\Kgrade(\fm_i,R_i/I_i)\\
&\leq\Kgrade(\fm_i,R_i)\\
&\leq \dim R_i.
\end{array}\]Thus, $$\{\Wdim R_i:i\in \Gamma\}\leq\sup\{\dim R_i:i\in \Gamma\}< \infty.$$ So, Lemma \ref{sy} completes the proof.
\end{proof}

\begin{proposition}\label{two} Let $R$ be a  quasilocal ring  with a desingularization $\{R_i:i\in \Gamma\}$.
The following holds:
\begin{enumerate}
\item[$\mathrm{(i)}$]   Any two-generated ideal of $R$ has   flat dimension bounded by  $1$.
\item[$\mathrm{(ii)}$]  If $\Gamma$ is countable, then any two-generated ideal  of $R$  has   projective dimension bounded by  $2$.
\end{enumerate}
 \end{proposition}

 \begin{proof}
$\mathrm{(i)}$: Let $I=(a,b)$ be a two-generated ideal of  $R$.
 Without loss of generality, we may assume that $R_i$ is local.
There is $i_0\in \Gamma$ such that  $\{a,b\}\subset R_i$ for all $i>i_0$. We now apply an idea of Buchsbaum-Eisenbud.
As, $R_i$ is an unique factorization domain and in view of \cite[Corollary 5.3]{BE}, $\{a,b\}$ has a greatest common divisor $c$
such that $\{a/c,b/c\}$ is a regular sequence. Thus, $\pd_{R_i}((a/c,b/c)R_i)<2$. Multiplication by $c$ shows that  $(a/c,b/c)\simeq(a,b)$.
Conclude that
$\pd_{R_i}((a,b)R_i)<2.$ Then by the same reasoning as Lemma  \ref{sy}(i), $\fd_R(I)<2$.

$\mathrm{(ii)}$: In view of part $\mathrm{(i)}$ the claim follows by the argument of Corollary \ref{coor}.
\end{proof}

We will use  the following result several times.

\begin{lemma}\label{kunz}(See \cite[Theorem 23.1]{Mat})
Let $\varphi$ be a local map from a  regular local ring $(R,\fm)$ to a  Cohen-Macaulay local ring $(S,\fn)$.
Suppose $\dim R+\dim S/\fm S=\dim S$. Then $\varphi$ is flat.
\end{lemma}

\begin{example}
The conclusion of Proposition \ref{two} can not carry over
three-generated ideals.
\end{example}

\begin{proof}
Let $k$ be any field. For each $n\geq 3$, set
$R_n:=k[x_1,\ldots,x_{2n-4} ]$. Define
 $$I_n:=(x_1,x_2)\cap \ldots\cap(x_{2n-5},x_{2n-4}),$$
$f_n:=x_1x_3\ldots x_{2n-5}$, and
$g_n:=x_2x_4\ldots x_{2n-4}$. Let $h_n$ be such that $((f_n,g_n):h_n)=I_n$.
It is proved in \cite{Bu} that $$\pd_{R_n}(R_n/(f_n,g_n,h_n))=n   \ \ (\ast)$$
 The assignments
 $x_1\mapsto x_1x_{2n-3}$, $x_2\mapsto x_2x_{2n-2}$, and
$x_i\mapsto x_{i}$ (for $i\neq1,2$)
defines the   ring homomorphism $\varphi_{n,n+1}:R_n\lo R_{n+1}.$ This has the following properties:
$\varphi_{n,n+1}(f_n)=f_{n+1}$,
$\varphi_{n,n+1}(g_n)=g_{n+1}$, and
$\varphi_{n,n+1}(I_n)\subseteq I_{n+1}$.
By using this we can choose $ h_{n+1}$ be  such that $\varphi_{n,n+1}(h_n)=h_{n+1} $.
Look at the directed system $\{R_n,\varphi_{n,n+1}\}$ and denote the natural
map from $R_n$ to $R:={\varinjlim}R_n$ by $\varphi_{n}$.
In view of the following commutative diagram,

$$\xymatrix{
& R_n\ar[r]^{\varphi_{n,n+1}}\ar[d]_{\varphi_{n}}&R_{n+1}\ar[dl]^{\varphi_{n+1}}\\
&R
 }$$the ideal $I:=(\varphi_{n}(f_n),\varphi_{n}(g_n),\varphi_{n}(h_n))R$
is independent of  $n$.

\begin{enumerate}
\item[Claim A.] The extension $R_n\to R_{n+1}$ is flat.
\item[Indeed,] denote the unique graded  maximal ideal of $R_n$ by $\fm_n$. Set $A_n:=(R_n)_{\fm_n} $. In view of \cite[Page 178]{Mat}, $R_n\to R_{n+1}$ is flat provided the induced map
$\psi_n:A_n\to A_{n+1}$ is flat.  In order to prove $\psi_n$ is flat, we note that
$\dim (A_n)=2n-4$ and $\dim (A_{n+1})=2n-2$. Let $\fm$ be the unique graded maximal ideal of $S:=k[x_1,x_2,x_{2n-3},x_{2n-2}]$. Then $\frac{A_{n+1}}{\fm_n A_{n+1}}\simeq \frac{S_\fm}{(x_1x_{2n-3},x_2x_{2n-2})}$. Since $x_2x_{2n-2}\notin\bigcup_{\fp\in\Ass\left(S_\fm/(x_1x_{2n-3})\right)}\fp$, the sequence $x_1x_{2n-3},x_2x_{2n-2}$
is  regular  over $S _\fm$. In particular, $\dim(\frac{A_{n+1}}{\fm_n A_{n+1}})=2.$
Thus, $\dim(A_{n+1})=\dim(A_n)+\dim(\frac{A_{n+1}}{\fm_nA_{n+1}}) $. In view of Lemma \ref{kunz} we observe that $A_n \to A_{n+1}$ is flat. This finishes the proof of the claim.
\end{enumerate}

Set $T_n:=\Tor_n^{R_n}(R_n/(f_n,g_n,h_n),k)$.
Due to $(\ast)$, $T_n\neq 0.$ Since $T_n$ is graded (see \cite[Page 33]{BH}), one has $(T_n)_{\fm_n}\neq 0$ (see e.g., \cite[Proposition 1.5.15(c)]{BH}). By \cite[Exercise 7.7]{Mat}, $\Tor$-modules compute with localization.
 In the light of the rigidity property of $\Tor$-modules over equal-characteristic regular local rings (Auslander–-Lichtenbaum) we see that $\Tor_{n-i}^{R_n}(R_n/(f_n,g_n,h_n),k)_{\fm_n}\neq 0$  for all $i\geq 0$. For more details, please see \cite{au}. In particular, $\Tor_{n-i}^{R_n}(R_n/(f_n,g_n,h_n),k)\neq 0$ for all $i\geq 0$.
The map $\varphi_{n,n+1}$ induces the following map:$$ \tau_{n,n+1}^{\ell}:\Tor_{\ell}^{R_n}(R_n/(f_n,g_n,h_n),R_{n}/ \fm_{n})\to  \Tor_{\ell}^{R_{n+1}}(R_{n+1}/(f_{n+1},g_{n+1},h_{n+1}),R_{n+1}/ \fm_{n}R_{n+1}).$$

\begin{enumerate}
\item[Claim B.] The map $\tau_{n,n+1}^{\ell}$ is one to one. \item[Indeed,]
in view of  Claim A, the extension $R_n\to R_{n+1}$ is flat. Set $T:=\Tor_{\ell}^{R_{n}}(R_{n}/(f_{n},g_{n},h_{n}),R_{n}/ \fm_{n})$. This is a graded module over $R_n$. Recall from \cite[Exercise 7.7]{Mat} that $$\Tor_{\ell}^{R_{n+1}}(R_{n+1}/(f_{n+1},g_{n+1},h_{n+1}),R_{n+1}/ \fm_{n}R_{n+1})\simeq T\otimes_{R_{n}} R_{n+1}.$$ Due to the proof of \cite[Theorem 7.4(i)]{Mat}, $T\hookrightarrow T\otimes_{R_{n}} R_{n+1} $  is one to one. Thus, $\tau_{n,n+1}^{\ell}$ is one to one. This finishes the proof of the claim.
\end{enumerate}

Let $\ell\leq n$. We combine Claim B  along with \cite[VI, Exercise 17]{CE} to construct the following injection

 \[\begin{array}{ll}
R_n/\fm_{\ell} R_n&\simeq (R_{\ell}/\fm_{\ell})\otimes_{R_{\ell}} R_n\\
&\hookrightarrow\Tor_{\ell}^{R_n}(R_n/(f_n,g_n,h_n),R_n/\fm_{\ell} R_n)\\
&\hookrightarrow \underset{i}{\varinjlim}\Tor_{\ell}^{R_n}(R_n/(f_n,g_n,h_n),R_n/\fm_{\ell} R_n)\\
&\simeq\Tor_{\ell}^{R}(R/I,R/\fm_{\ell}R).
\end{array}\]
Thus,
$0\neq  R/\fm_{\ell}R\hookrightarrow\Tor_{\ell}^{R}(R/I,R/\fm_{\ell}R)$. So  $\pd_R(R/I)=\infty,$ as claimed.
\end{proof}

\section{ Desingularization via super regularity}

We will use the following result several times.

\begin{lemma}\label{sr}(See \cite{V1})
Let $(R,\fm)$ be a  super regular ring. Then $\fm$ can be generated by a regular sequence.
In particular, $\fm$ is finitely generated.
\end{lemma}

The notation $\sym_{R}(-)$ stands for the symmetric algebra of an $R$-module.
Also, we set
$\Gr_R(I):=\bigoplus _{i=0}^{\infty}I^i/I^{i+1},$ where $I$ is an ideal of  $R$.

\begin{lemma}\label{gr}
Let $\{(R_i,\fm_i,k_i):i\in \Gamma\}$ be a directed system of  local rings.  Set $R:=\varinjlim R_i$, $\fm=\varinjlim \fm_i$ and $k=\varinjlim k_i$.
The following holds:
\begin{enumerate}
\item[$i)$] $\Gr_{R}(\fm)\simeq{\varinjlim}_i \Gr_{R_i}(\fm_i)$.
\item[$ii)$] $\sym_k(k^{\oplus{\mu(\fm)}})\simeq{\varinjlim}_i \sym_{k_i}( k_i^{\oplus{\mu(\fm_i)}}) $.
\end{enumerate}
\end{lemma}

\begin{proof}
$i)$ Taking  colimit
of the following exact sequence of directed systems
$$0\longrightarrow \{\fm_i^{n+1}\}_i\longrightarrow\{\fm_i^{n}\}_i\longrightarrow\{\fm_i^{n+1}/ \fm_i^{n}\}_i\longrightarrow0,$$ and using 5-lemma,
yields that  ${\varinjlim}_i \fm_i^n/ \fm_i^{n+1}\simeq\fm^n/ \fm ^{n+1}$. In particular, $\Gr_R(\fm)\simeq{\varinjlim}_i \Gr_{R_i}(\fm_i)$.

$ii)$ This is in \cite[8.3.3]{G}.
 \end{proof}

\begin{proposition}\label{kab}
 Let $\{(R_i,\fm_i):i\in \Gamma\}$ be a directed system of   local rings with the property that
 $\fm_i^2=\fm_i\cap \fm_{i+1}^2$. If $R:={\varinjlim}R_i$ is coherent and super regular, then  each $R_i$  is regular.
\end{proposition}

\begin{proof}
Denote the maximal ideal of $R$ by $\fm$ and denote the residue field of $R$ (resp. $R_i$) by $k$ (resp. $k_i$).
In view of Lemma \ref{sr}, $\fm$ is generated by a regular sequence. Thus things equipped
with the following isomorphism
$$
\begin{CD}
\theta:\sym_{k}(k^{\oplus {\mu(\fm )}}) @>>> \Gr_R(\fm):=\bigoplus _{i=0}^{\infty}\fm^i/\fm^{i+1}.\\
\end{CD}
$$
 Look at the natural epimorphism$$
\begin{CD}
\theta_i:\sym_{k_i}( k_i^{\oplus{\mu(\fm_i)}}) \twoheadrightarrow \Gr_{R_i}(\fm_i),\\
\end{CD}
$$
and the natural map $\varphi_i:V_i:=\fm_i/ \fm_i^2\hookrightarrow V_{i+1}:=\fm_{i+1}/ \fm_{i+1}^2.$
We claim that:

Claim A. The map $\vartheta_{i}:=\sym(\varphi_i)$ is monomorphism.

Indeed,
we look at the following diagram:
$$\xymatrix{
 \sym_{k_i}(V_i)\ar[dr]_{\vartheta_{i}}\stackrel{f}\hookrightarrow&\sym_{k_i}(V_{i+1})\stackrel{g}\hookrightarrow
 \sym_{k_i}(V_{i+1})\otimes_{k_i}k_{i+1}\ar[d]_{h}\ar[r]^{\  \   \  \ \simeq}&\sym_{k_{i+1}}
 (V_{i+1}\otimes_{k_i}k_{i+1})\ar[d]_{  \simeq}\\
                               &\sym_{k_{i+1}}(V_{i+1})\ar[r]^{i}& \sym_{k_{i+1}}(\bigoplus_{\dim_{k_i}(k_{i+1})} V_{i+1}) \\
                    & & & &}$$Remark that
\begin{enumerate}
\item[$1)$] Since $V_{i}$ is a direct summand of $V_{i+1}$ as a $k_i$-vector space, $f$ is a monomorphism.
\item[$2)$] The map $g$ is monomorphism,  because $k_i$ is a field.
\item[$3)$] The horizontal isomorphism follows by \cite[8.3.2]{G}.
\item[$4)$] The vertical isomorphism follows by $\bigoplus_{\dim_{k_i}(k_{i+1})}  V_{i+1}\simeq  V_{i+1}\otimes_{k_i}k_{i+1}$.
\item[$5)$] Since $V_{i+1}$ is a direct summand of $\bigoplus V_{i+1}$ as a $k_{i+1}$-vector space, $i$ is a
monomorphism.
\end{enumerate}
 By these, we conclude that the map $h$ is a monomorphism.
So $\vartheta_{i}:\sym_{k_i}(V_i)\to \sym_{k_{i+1}}(V_{i+1})$
is  monomorphism. This completes the proof of Claim A.

Set $K_i:=\ker\theta_i$. Also, remark that  $\vartheta_{i}(K_i)\subseteq K_{i+1}$.
Again we denote the restriction map by $\vartheta_{i}:K_i\hookrightarrow K_{i+1}.$
Recall that ${\varinjlim} k_i\simeq k$ and ${\varinjlim}(\fm_i^n/\fm_i^{n+1})\simeq\fm ^n/\fm ^{n+1}$. In view of Lemma \ref{gr},
 $\sym(-)$ and $\Gr(-)$ behave well with respect to
direct limits.
 Hence, $\theta={\varinjlim} \theta_i$. Put all of these together to observe $$ K_i\hookrightarrow{\varinjlim} K_i\simeq\ker \theta=0.$$
 So, $\theta_i$ is an isomorphism. This means that $\fm_i$ is generated by a regular sequence. The regularity of $R_i$ follows by this, because
$R_i$ is noetherian.
\end{proof}

By $\gd(R)$  we mean the \textit{global dimension}  of $R$.
Let $R$ be a  noetherian local domain of prime characteristic $p$. Recall that the  perfect closure  of  $R$ is defined by adjoining to $R$ all higher $p$-power
 roots of all elements of $R$ and denote it by $R^\infty$.

\begin{proposition}\label{rem}
Let $R$ be a local  domain  of prime characteristic which is either excellent or homomorphic
image of a Gorenstein local ring and suppose that its perfect closure is coherent (e.g., $R$ is regular).
 If $R$ is not a field, then $\gd(R^{\infty})=\dim R+1$.
\end{proposition}

\begin{proof}
Let $\underline{x}:=x_1,\ldots,x_d$ be a system of parameters for $R$ and set $p:=\Char R$.
\begin{enumerate}
\item[Claim A.]  One has $\underline{x}$ is a regular sequence on $R^{\infty}$.
\item[  Indeed,] this is  in \cite[Theorem 3.10]{Sh} when  $R$ is homomorphic image of a Gorenstein ring. The argument of \cite[Theorem 3.10]{Sh} is based on \emph{Almost Ring Theory}. The claim in the excellent case is  in \cite[Lemma 3.1]{ab}. This uses \textit{non-noetherian Tight Closure Theory}.
\end{enumerate}  In particular, $\fd(R^{\infty}/ \underline{x}R^{\infty})=d.$ Combining this with  \cite[Theorem 1.2]{A}, $\Wdim(R^{\infty})= \dim R<\infty.$
The same citation yields that $\gd(R^{\infty})\leq\dim R+1$.  Suppose on the contrary that $$\gd(R^{\infty})\neq\dim R+1.$$
 This says that  $\gd(R^{\infty})=\Wdim R$. Note that $R^{\infty}$ is coherent and quasilocal. Denote its maximal ideal by $\fm_{R^{\infty}}$.
By definition, $R^{\infty}$ is  super regular.
 In the light of Lemma \ref{sr},  $\fm_{R^{\infty}}$ is finitely generated. We bring the following claim.
 \begin{enumerate}
\item[Claim B.] One has $\fm_{R^{\infty}}=\fm_{R^{\infty}}^p$.
\item[  Indeed,] clearly, $\fm_{R^{\infty}}^p\subset\fm_{R^{\infty}}.$ Conversely, let $r\in\fm_{R^{\infty}}$.
Since $R$ is perfect, any polynomial such as $f(X):=X^p-r$ has a root. Let $r^{1/p}\in R^{\infty}$ be a root of $f$. We have $(r^{1/p})^p \in\fm_{R^{\infty}}$. Since  $\fm_{R^{\infty}}$ is prime, $r^{1/p}\in \fm_{R^{\infty}}$. We conclude from this that  $r=(r^{1/p})^p \in\fm_{R^{\infty}}^p$, as claimed.
\end{enumerate}
By  Nakayama's Lemma, $\fm_{R^{\infty}}=0$, i.e., $ R^{\infty}  $ is a field. So, $R$ is a field.
This is a contradiction.
\end{proof}

\begin{question}Let $R$  be a  local  domain  of prime characteristic.
 What is $\gd( R^{\infty}) $?
\end{question}

\begin{proposition}\label{cpure}
Let $R$ be a quasilocal containing a field of prime characteristic $p$ which is integral and purely inseparable extension  of an
 $F$-finite regular local ring $R_0$. If $R$ contains
all roots of  $R_0$,
then $R$ has a desingularization with respect to a flat directed system of noetherian regular rings.
\end{proposition}

\begin{proof}
Note that
$R_0$ contains a field. Write $R$ as a directed union of a filtered system  $\{R_i\}$ of its subrings which are finitely generated algebras over $R_0$.

\begin{enumerate}
\item[Claim A.]  The ring $R_i$ is local.
\item[  Indeed,]  by assumption $R_0$ is local. Denote its maximal ideal by $\fm_0$. Let $\fm_i$ and $\fn_i$ be two maximal ideals of $R_i$.
 Both of them lying over $\fm_0$.
Let $r\in\fm_i$. Since the extension $R_0\to R_i$ is integral and purely inseparable, we observe that $r^{p^n}\in R_0$ for some $n\in \mathbb{N}$.
Thus $r^{p^n}\in \fm_0=\fn_i\cap R_0$. Since $\fn_i$ is prime and $r\in R_i$, we deduce that $r\in\fn_i$. Hence $\fm_i\subset \fn_i$.
Therefore  $\fm_i= \fn_i$, because $\fm_i$ is maximal. So, $R_i$ is local,
 as claimed.
\end{enumerate}

We denote the unique maximal ideal of $R_i$ by $\fm_i$.
Since $R_0\to R_1$ is integral, $d:=\dim R_0=\dim R_1$.
Remark that if $y\in R_1$, there is $n_1\in \mathbb{N}$ such that $y^{p^{n_1}}\in R_0$.
Since $R_0\to R_1$ is integral, $ R_1$ is finitely generated as an $R_0$-module. From this we can pick a uniform $n$ such that for any $y\in R_1$,
 $y^{p^{n}}\in R_0$. After adding $R_0^{1/p^{n}}$ to $R_1$ and denoting the new ring again by $R_1$, we may assume that $R_0^{1/p^{n}}\subset R_1$,
 here is a place that we use the assumptions $R_0^{\infty}\subseteq R$ and that $R_0^{1/p^{n}}$ is finite over $R_0$.
Now, let $\underline{x}$ be a minimal generating set for $\fm_0$. In particular, $\underline{x}$  is a regular system of parameters on $R_0$.

 \begin{enumerate}
\item[Claim B.] Let $n$ be as the above paragraph. Then $\fm_1=(x_1^{1/p^{n}},\ldots,x_d^{1/p^{n}}) R_1$.
\item[  Indeed,]  let $y\in \fm_1$. Then   $y^{p^n}\in R_0$. In particular, $y^{p^n}\in \fm_0$. Then   $y^{p^n}=\sum r_ix_i$ where $r_i\in R_0$.
 We are in a situation to take $p^n$-th root in $R_1$,  this is due to the choose of $n$. Taking $p^n$-th roots, we have $y=\sum r_i^{1/p^{n}}x_i^{1/p^{n}}$, where $r_i^{1/p^{n}}\in R_1$ and $x_i^{1/p^{n}}\in \fm_1$.
Thus $y\in(x_1^{1/p^{n}},\ldots,x_d^{1/p^{n}}) R_1$. Therefore, $\fm_1\subset(x_1^{1/p^{n}},\ldots,x_d^{1/p^{n}})R_1\subsetneqq R_1$.  The reverse inclusion is trivial, because $\fm_1$ is maximal. So,  $\fm_1=(x_1^{1/p^{n}},\ldots,x_d^{1/p^{n}}) R_1$
 as claimed.
\end{enumerate}

In view of the claim,  $R_1$ is regular.
Since $\fm_0 R_1$   is   $\fm_1$-primary, the extension $R_0\to R_1$ is flat, please see Lemma \ref{kunz}.
Repeating this, one may observe that $\{R_i\}$ is a desingularization for $R$
and that $R_i\to R_j$ is  flat.
\end{proof}

\begin{lemma}\label{abo}
One has ${\varinjlim}_{i\in\Gamma}R_i[X]
 \simeq({\varinjlim}_{i\in\Gamma}R_i)[X]$.
\end{lemma}

 \begin{proof}
This is straightforward and we leave it to the reader.
\end{proof}

\begin{corollary}
Adopt the notation of Proposition \ref{cpure}. Then
$R$ is stably coherent.
\end{corollary}

\begin{proof}
There is a flat directed system $\{R_i\}$ of noetherian regular rings such that its direct limit is $R$. In particular, $R_i[X]\to R_j[X]$ is  flat.
In view of \cite[Theorem 2.3.3]{G} and Lemma \ref{abo}, $R$ is stably coherent.
\end{proof}
The assumption $R_0^{\infty}\subseteq R$ in Proposition \ref{cpure} is really needed:

\begin{example}\label{e}
Let $F$ be a field of characteristic $2$ with $[F:F^2] = \infty$. Let $\hat{R_0 }= F[[x, y]]$ be the formal power series on variables $\{x,y\}$ and look at
$R_0 := F^2[[x, y]][F]$. Let $\{b_i:i\in \mathbb{N}\}\subset F$ be an infinite set of
$2$-independent elements. Set \begin{enumerate}
\item[] $e_n := \frac{\sum_{i=n}^{\infty} (xy)^ib_i}{y^n}$,
\item[] $f_n := \frac{\sum_{i=n}^{\infty} (xy)^ib_i}{x^n}$.
\end{enumerate} Define
$R := R_0[e_i,f_i:i\in \mathbb{N}] $. This is quasilocal. Denote its unique maximal ideal by $\fm$.
Recall from  \cite[Page 206]{N} that  $R_0\to \hat{R_0 }$ is integral and purely inseparable.  Since
$R_0\subset R \subset\hat{R_0 }$, we get that $R_0\subset R$ is integral and purely inseparable.
By \cite[Example 1]{K},
$\pd_R(\fm)=\infty$
and that $\fm$
is finitely generated. In particular, $R$ is not
regular. We conclude  from Proposition \ref{sy2} that
$R$ is not coherent. So, $R$ has no  desingularization with respect to its noetherian regular subrings.
\end{example}

\section{ Desingularization via purity}
We begin by recalling the notion of the \textit{purity}.
Let $M\subset N$ be  modules over a ring $R$. Recall that   $M$ is pure in $N$ if $M\otimes_R L\to N\otimes_R L $ is monomorphism
for every $R$-module $L$. We say a directed system $\{R_i:i\in \Gamma\}$ is   pure if  $R_i\longrightarrow R_j$ is pure  for all $i,j\in\Gamma$ with $i\leq j$.

\begin{proposition}\label{cri}
 Let $\{(R_i,\fm_i):i\in \Gamma\}$ be a pure directed system of local rings and  such that the maximal ideal of
$(R,\fm):={\varinjlim}_{i\in I}R_i$  has a finite free resolution.
 Then the following assertions are true:
 \begin{enumerate}
\item[$\mathrm{(i)}$]  There exists an $i\in \Gamma$ such that $R_j$ is regular for all $i \leq j$.
\item[$\mathrm{(ii)}$] There exists an $i\in \Gamma$ such that $R_j\to R_k$ is flat for all $i \leq j\leq k$.
\item[$\mathrm{(iii)}$]  The ring $R$ is noetherian and regular.
\end{enumerate}
\end{proposition}

\begin{proof}
$\mathrm{(i)}$: Look at the following
finite free resolution of $\fm$:$$\xymatrix{0\ar[r]&F_{N}\ar[r]&\ldots\ar[r]&F_{j+1}\ar[r]^{f_j}&F_{j}\ar[r]&\ldots\ar[r]&F_{0}\ar[r]&\fm\ar[r]&0,}$$
where $F_{j}$ is finite free and $f_j$ is given by a matrix with  finite rows and finite columns.
By $I_{t}(f_j)$, we mean the ideal
generated by $t\times t$ minors of  $f_j$.
 Let $r_j$ be the expected rank of $f_j$, see  \cite[Section 9.1]{BH} for its definition.
 By \cite[Theorem 9.1.6]{BH},  $\Kgrade_R(I_{r_j}(f_j), R)\geq j.$
There is an index $i\in \Gamma$ such that all of components of  $\{f_j\}$ are in $R_i$. Let $F_{j}(i)$ be the
free $R_i$-module with the same rank as $F_{j}$. Consider $f_j$ as a matrix over $R_i$, and denote it by $f_j(i)$.
Recall that $\fm$ is finitely generated. Choosing $i$ sufficiently large,  we may assume that
$\fm=\fm_i R$.   For the simplicity of the reader, we bring the following claim.
\begin{enumerate}
\item[Claim A.]  Let $A$ be a subring of a commutative ring $B$. Let $X\in \Mat_{rs}(A)$ and $Y\in \Mat_{st}(A)$ be  matrices.  Look at $X\in \Mat_{rs}(B)$ and $Y\in \Mat_{st}(B)$ as  matrices whose entries coming from  $B$. If $XY\in\Mat_{rt}(B)$ is zero as  a matrix over $B $, then $XY\in\Mat_{rt}(A)$ is zero as a matrix over $A $.
\item[Indeed,]  since $A$ is a subring of $B$, the claim is trivial.
\end{enumerate}

Thus, $f_j(i)f_{j+1}(i)=0$.
Look at the following
complex of finite free  modules:$$\xymatrix{0\ar[r]&F_{N}(i)\ar[r]&\ldots\ar[r]&F_{j+1}(i)\ar[r]^{f_j(i)}&F_{j}(i)
\ar[r]&\ldots\ar[r]&F_{0}(i)\ar[r]&\fm_i\ar[r]&0.} $$
We are going to show that this is exact.  Recall that $I_{t}(f_j(i))$ is the ideal
generated by $t\times t$ minors of $f_j(i)$. Clearly,  $r_j$ is
the expected rank of $f_j(i)$. Let $\underline{z}:=z_1,\ldots,z_{s}$ be a generating set
for $I_{t}(f_j(i))$.  In view of the purity, there are monomorphisms $0\lo H_j
(\underline{z},R_i)\lo H_j
(\underline{z},R)$ for
all $i$ and $j$, see \cite[Exercise
10.3.31]{BH}. Then,
$$\Kgrade_R(I_{r_j}(f_j), R)\leq\Kgrade_{R_i}(I_{r_j}(f_j(i) ),R_i).$$
Thus, $\Kgrade_{R_i}(I_{r_j}(f_j), R_i)\geq j.$
Again, due to  \cite[Theorem 9.1.6]{BH}, $$0\lo F_{N}(i)\lo\cdots\lo F_{0}(i)$$ is acyclic. Thus, $\pd(R_i/\fm_i)<\infty$.
By Local-Global-Principle (please see \cite[Theorem 2.2.7]{BH}),  $R_i$ is  regular.

$\mathrm{(ii)}$:   By purity, $\dim R_m\geq \dim R_n$ for all $n\leq m$, see \cite[Remark 4 and Corollary 5]{Br}.
 Again, in the light of purity,$$\fm_m =(\fm_m R)\cap R_m =(\fm_n R)\cap R_m = (\fm_n R_m)R\cap R_m = \fm_n R_m,$$
    for all $n\leq m$.
Thus $\fm_n R_m=\fm_m$. Denote the  minimal number of elements required to generate  the ideal $\fm_m$ by $\mu(\fm_m)$. Consequently,  $\mu(\fm_m)\leq\mu(\fm_n)$. By part  $\mathrm{(i)}$,
$(R_i,\fm_i)$ is regular. Hence $$\dim R_m=\mu(\fm_m)\leq\mu(\fm_n)= \dim R_n.$$ Therefore, $\dim R_m= \dim R_n$. In view of Lemma \ref{kunz},
$R_n\to R_m $ is  flat     for all  $n\leq m$, as claimed.

$\mathrm{(iii)}$: Recall from $\mathrm{(ii)}$ that $\fm_m = \fm_n R_m$ for all $n<m$. In view of
\cite{O}, $R$ is noetherian. The ring $R$ is regular, because $\pd_R(R/\fm)<\infty$.
\end{proof}

\begin{example}\label{ex}Here, we present a desingularization with respect to  a non-pure directed system. To this end
let $R:=\{n+\sum_{i=1}^{\ell} n_i t^i:n\in \mathbb{Z}, n_i\in \mathbb{Z}[1/2]\}$.
Then $R$ has a desingularization $\{(R_i,\phi _{i,j})\}$ where $\phi _{i,j}:R_i\to R_j$ is not pure.
\end{example}

\begin{proof}
For each $i\in \mathbb{N}$, set $R_i:=\mathbb{Z}[t/2^i]$. Note that  $R_i $ is a noetherian regular ring.
The system  $\{R_i\}_{i\in \mathbb{N}}$  is  directed with respect to the inclusion.
 Let $f\in R$. Then $f=n+\sum_{i=1}^{\ell} n_i t^i$ where $n_i\in \mathbb{Z}[1/2]$. There is $k\in\mathbb{Z}$ such that $n_i=m_i/2^{ik}$
 for some $m_i\in \mathbb{Z}$ and for all $i$. Deduce by this that $f\in R_{k}$. Thus,  $\{R_i\}$ gives a desingularization for $R$.
Now, we look at the following equation  $2\textbf{X}=t/2^i.$ Clearly,
$t/2^{i+1}\in R_{i+1}$
is a solution.
The equation has no solution in  $R_{i}$. In the light of  \cite[Theorem 7.13]{Mat}, the map  $R_{i}\to R_{i+1}$ is not pure, as claimed.
\end{proof}

Also, we present the following example.

\begin{example}
Let $D$ be a noetherian regular integral domain  with a fraction field $Q$. Let $R:=\{f\in Q[X_1,\ldots,X_n]:f(0,\ldots,0)\in D\}$.
Then $R$ has a desingularization.
\end{example}

\begin{proof}
Without loss of the generality we may assume that $D\neq Q$.
Recall that $$R=\{f\in Q[X_1,\ldots,X_{n}]:f(0,\ldots,0)\in D\}\simeq\bigotimes_{1\leq i \leq n}\{f\in Q[X_i]:f(0)\in D\}.$$
 Let $F:=Q^{\oplus n}$. This  is flat as a $D$-module.
 Under the identification $\sym_Q(Q)=Q[X]$, the image of $\sym_{D}(Q)$ in the natural map
$\sym_{D}(Q)\to\sym_Q(Q)$
is $D+XQ[X]$. So
\[\begin{array}{ll}
\sym_{D}(F)&\simeq\bigotimes_{n}\sym_{D}(Q)\\
&\simeq\bigotimes_{n}\{f\in Q[X_i]:f(0)\in D\}.
\end{array}\]
Due to the Lazard's theorem, there is a directed system of finitely generated free modules $\{F_i:i\in \Gamma\}$
with direct limit $F$.
In view of \cite[8.3.3]{G},
$$R\simeq\sym_{D}(F)\simeq{\varinjlim}_{i \in \Gamma} \sym_{D}(F_i).$$Since $R_i:=\sym_{D}(F_i)$ is a noetherian regular ring, $R$ can be realized as a direct limit of the directed system $\{R_i:i\in \Gamma\}$ of noetherian
regular rings.
\end{proof}

\section{  Desingularization of products}

The following definition is taken from \cite{MR}.

\begin{definition}
A ring is called \textit{DLFPF}, if it is a direct limit of a finite product of fields.
\end{definition}

\begin{question}\label{q}(See \cite[Question 10]{MR})
Let $E$ be a field and let $R$ be a maximal \textit{DLFPF} subring
of $\prod_{\mathbb{N}}E$. Does $R$ contain a field isomorphic to $E$?
\end{question}

The
 following answers Question \ref{q} in the finite-field case.

\begin{proposition}\label{pro}
Let $F$ be a finite field. Then  $\prod_{\mathbb{N}} F\simeq {\varinjlim}_{}(\bigoplus_{finite}F_i)$
 where $F_i$ is a field.
 \end{proposition}

 \begin{proof}
 Let $S\subset \prod_{\mathbb{N}} F$ be the subring consisting
of all elements that have only finitely many distinct coordinates.
  Since $F$ is finite,  $S= \prod_{\mathbb{N}} F$.
 By \cite[Proposition 5.2]{GH},
$S=\bigcup A_j$  where   $A_j$  is an artinian regular subring of $\prod_{\mathbb{N}} F$.
It remains to note that  any artinian regular ring is isomorphic to a finite direct product of fields.
\end{proof}

\begin{corollary}\label{prod}
 Let $F$ be a finite field. Then
 $\prod_{\mathbb{N}} F$ is stably coherent.
 \end{corollary}

 \begin{proof}
 We adopt the notation of Proposition \ref{pro} and we denote a finite family of variables by $\underline{X}$.
Let $R_{\gamma}:=(\bigoplus_{1\leq i \leq n_{\gamma}}F_i)$.
Let $\gamma\leq \delta$.
Since $R_{\gamma}$ is of zero weak dimension,
$R_{\gamma}\to R_{\delta}$ is flat.
 It turns out that $R_{\gamma}[\underline{X}]\to R_{\delta}[\underline{X}]$ is flat.
In view of Lemma \ref{abo} and by Proposition \ref{pro},  $$
(\prod F)[\underline{X}]
\simeq({\varinjlim}_{\gamma\in\Gamma}R_{\gamma})[\underline{X}]
\simeq {\varinjlim}_{\gamma\in\Gamma}(R_\gamma[\underline{X}]).$$ is a flat direct limit of  noetherian  regular rings.
By \cite[Theorem 2.3.3]{G},  $\prod_{\mathbb{N}} F$  is stably coherent.
\end{proof}

In the following item we collect some homological properties of the products of rings.

\begin{fact}\label{so}  i) Let $R$ be a noetherian local ring.
Combining \cite[Theorem 6.1.2]{G} and \cite[Theorem 6.1.20]{G}, yields that
$\prod_{\mathbb{N}} R$ is coherent if and only if
$\dim R<3$.

ii)
Let $\{R_N:n\in \mathbb{N}\}$ be a family of rings such that $R:=\prod R_n$ is
coherent. Then $\Wdim R = \sup \{ \Wdim R \}$, please see \cite[Theorem 6.3.6]{G}.

iii) (See the proof of \cite[Corollary 2.47]{O1}) Let  $\fa$ be an ideal of a $\aleph_n$-noetherian ring
$A$. Then $\pd_{A}(A/\fa)\leq\fd_{A}(A/\fa)+n+1$.

iv) (See \cite[Introduction]{MR}) The ring  $\prod_\mathbb{N}\mathbb{Q}$
is not written as a direct limit of a finite product of fields.
 \end{fact}

\begin{example}\label{non}
The  ring $R:=\prod_\mathbb{N}\mathbb{Q}$
 is coherent and regular. But, $R$ has no desingularization with respect to its noetherian regular subrings.
\end{example}

\begin{proof}
By Fact \ref{so} i),
$\prod_\mathbb{N}\mathbb{Q}$ is coherent.
In the light of  Fact \ref{so} ii), $R$ is \textit{von Neumann  regular}, i.e., $R$ is of zero weak dimension.
 We deduce from Fact \ref{so} iii) that $\gd(R)<3$. Therefore, $R$
 is coherent and regular.
Suppose on the contrary that $R$ can be written as a direct limit of its noetherian regular subrings $\{R_i:i\in I\}$. Due to \cite[Corollary 2.2.20]{BH}, $R_i\simeq R_{i1}\times\ldots \times R_{in_{i}}$, where
$R_{ij}$ is a noetherian regular domain.  By $Q_{ij}$ we mean the fraction field of $R_{ij}$. Recall that $\varinjlim _{i}\varinjlim _{j}Q_{ij}$
is a direct limit of finite product of its subfields.
This is a  consequence of the fact that any double direct limit is a direct limit, please see \cite[III) Proposition 9]{bo}.
Recall from \cite[Corollary 4]{MR} that a von Neumann  regular subring of
$DLFPF$  is $DLFPF$. We conclude by this that
$\prod_\mathbb{N}\mathbb{Q}$ can be written as a direct limit of a finite product of fields.
This is a contradiction, please see Fact \ref{so} iv).
\end{proof}

\begin{acknowledgement}
I would like to thank M. Dorreh, S. M. Bhatwadekar and H. Brenner, because of a talk on direct limits.
I thank  the anonymous referee for various suggestions.
\end{acknowledgement}


\end{document}